\documentclass[11pt]{article}
\usepackage{graphicx}
\usepackage[centertags]{amsmath}
\usepackage{amsfonts}
\usepackage{amssymb}
\usepackage{amsthm}
\usepackage{newlfont}
\newtheorem{thm}{Theorem}[section]
\newtheorem{cor}[thm]{Corollary}

\newtheorem{prop}[thm]{Proposition}
\theoremstyle{definition}
\newtheorem{defn}[thm]{Definition}
\theoremstyle{remark}

\newcommand{\grad}{\textrm{grad}}
\newcommand{\id}{\textrm{id}}

\title{Riemannian manifolds with two circulant structures}

\author{Iva Dokuzova}

\begin{document}

\maketitle
%

\begin{abstract}
We consider a  three-dimensional Riemannian manifold equipped with two circulant structures -- a metric $g$ and a structure $q$, which is an isometry with respect to $g$ and the third power of $q$ is minus identity. We discuss some curvature properties of this manifold, we give an example of such a manifold and find a condition for $q$ to be parallel with respect to the Riemannian connection of $g$.
\end{abstract}



\section*{Introduction}

 Important role in the geometry of the differentiable manifolds play additional structures. Often the structure satisfies an equation of second power, for example \cite{1}, \cite{3}, \cite{4}, but rarely it satisfies an equation of another power \cite{5}, \cite{10}, \cite{6}, \cite{7}. On the other hand the application of circulative structures in the geometry is interesting. In \cite{9} a geometric application of circulant matrices was obtained.

In this paper we consider a three-dimensional differentiable manifold $M$ with a Riemannian metric $g$ whose matrix in local coordinates is a special circulant matrix. Moreover, we consider an additional structure $q$ on $M$ with $q^{3}=-\id$ such that its matrix in local coordinates is also circulant. It turns out that $q$ is an isometry with respect to $g$. We note that there exists a local coordinate system, where the matrices of the components of these structures are circulant \cite{8}.

The paper is organized as follows. In Sect. \ref{1} we introduce on a three-dimensional differentiable manifold $M$ a Riemannian metric $g$ whose matrix in local coordinates is a special circulant matrix. Furthermore, we consider an additional structure $q$ on $M$ with $q^{3}=-\id$ such that its matrix in local coordinates is also circulant. Thus, the structure $q$ is an isometry with respect to $g$.
We denote by $(M, g, q)$ the manifold $M$ equipped with the metric $g$ and the structure $q$.
In Sect. \ref{2} we obtain the conditions under which an orthogonal basis of type $\{x, qx, q^{2}x\}$ exists in the tangent space of a manifold $(M, g, q)$. The main result there is Theorem \ref{t3.1}. We give an example of such an orthogonal basis. In Sect. \ref{3} we establish relations between the sectional curvatures of some special $2$-planes in the tangent space. In Sect. \ref{4} we obtain a necessary and sufficient condition for $q$ to be parallel with respect to the Riemannian connection of $g$. Finally, in Sect. \ref{5}, we give an example of the considered manifold $(M, g, q)$.


\section{Preliminaries}\label{1}
Let $M$ be a three-dimensional manifold with a Riemannian metric $g$. Let the matrix of the local components of the metric $g$ at an arbitrary point $p(X^{1}, X^{2}, X^{3})\in M$ be a circulant matrix of the following form:
\begin{equation}\label{f2}
    (g_{ij})=\begin{pmatrix}
      A & B & B \\
      B & A & B \\
      B & B & A \\
    \end{pmatrix},
\end{equation}
where $A$ and $B$ are smooth functions of $X^{1}, X^{2}, X^{3}$. We will suppose that
\begin{equation}\label{ab}
    A>B>0.
\end{equation}
Then the conditions to be a positive definite metric $g$
are satisfied:

\begin{equation*}
    A>0,\quad \begin{vmatrix}
      A & B\\
      B & A \\
    \end{vmatrix}=(A-B)(A+B)>0,
\end{equation*}
\begin{equation*}
     \begin{vmatrix}
      A & B & B \\
      B & A & B \\
      B & B & A \\
    \end{vmatrix}=(A-B)^{2}(A+2B)>0.
\end{equation*}

We denote by $(M, g)$ the manifold $M$ equipped with the Riemannian metric $g$ defined by (\ref{f2}) with conditions (\ref{ab}).

Let $q$ be an endomorphism in the tangent space $T_{p}M$ of the manifold $(M,g)$ with
\begin{equation}\label{1.3}
    q^{3}=-\id,\qquad q\neq-\id,
\end{equation}
where $\id$ is the identity in $T_{p}M$. We suppose the local coordinates of $q$ are given by the circulant matrix
\begin{equation*}
    (q_{i}^{j})=\begin{pmatrix}
      a_{1} & a_{2} & a_{3} \\
      a_{3} & a_{1} & a_{2} \\
      a_{2} & a_{3} & a_{1} \\
    \end{pmatrix},\quad (a_{1},\ a_{2},\ a_{3}\ \in \mathbb{R}).
\end{equation*}
 Due to (\ref{1.3}) we have $(q_{i}^{.j})^{3}=-E,\ (q_{i}^{.j})\neq -E$ ($E$ is the unit matrix). Then we get the following system for $a_{1}, a_{2}, a_{3}$:
\begin{align*}
    &a_{1}^{3}+a_{2}^{3}+a_{3}^{3}+6a_{1}a_{2}a_{3}=-1,\\
    &a_{1}^{2}a_{2}+a_{1}a_{3}^{2}+a_{2}^{2}a_{3}=0,\\
    &a_{1}a_{2}^{2}+a_{3}a_{1}^{2}+a_{3}^{2}a_{2}=0,
\end{align*}
where $(a_{1}, a_{2}, a_{3})\neq (-1,0, 0)$.
Hence we get two cases for $q$:
\begin{equation*}
\begin{pmatrix}
      0 & -1 & 0 \\
      0 & 0 & -1 \\
      -1 & 0 & 0 \\
    \end{pmatrix},\quad\begin{pmatrix}
      0 & 0 & -1 \\
      -1 & 0 & 0 \\
      0 & -1 & 0 \\
    \end{pmatrix}.
\end{equation*}
In the further studies of both cases similar results are obtained. We choose the case
\begin{equation}\label{f4}
    (q_{i}^{.j})=\begin{pmatrix}
      0 & -1 & 0 \\
      0 & 0 & -1 \\
      -1 & 0 & 0 \\
    \end{pmatrix}.
\end{equation}

We denote by $(M, g, q)$ the manifold $(M,g)$ equipped with the structure $q$ defined by (\ref{f4}).

Further, $x, y, z, u$ will stand for arbitrary elements of the algebra on the smooth vector fields on $M$ or vectors in the tangent space $T_{p}M$. The Einstein summation convention is used, the range of the summation indices being always $\{1, 2, 3\}$.

From (\ref{f2}) and (\ref{f4}) we get immediately the following
\begin{prop} The structure $q$ of the manifold $(M,g,q)$ is an isometry with respect to the metric $g$, i.e.
\begin{equation}\label{2.1}
    g(qx, qy)=g(x, y).
\end{equation}

\end{prop}
\section{Orthogonal $q$-bases of $T_{p}M$}\label{2}

\begin{defn}
A basis of type $\{x, qx, q^{2}x\}$ of $T_{p}M$ is called a $q$-\textit{basis}. In this case we say that \textit{the vector $x$ induces a $q$-basis of} $T_{p}M$.
\end{defn}
If $x=(x^{1},x^{2},x^{3})\in T_{p}M$, then $qx=(-x^{2},-x^{3},-x^{1})$, $q^{2}x=(x^{3},x^{1},x^{2})$.
Obviously, we have the following
\begin{prop}
A vector $x=(x^{1},x^{2},x^{3})$ induces a $q$-basis of $T_{p}M$ if and only if
   \begin{equation}\label{lema2}
   3x^{1}x^{2}x^{3}\neq (x^{1})^{3}+(x^{2})^{3}+(x^{3})^{3}.
   \end{equation}
\end{prop}
In our next considerations we use an orthogonal $q$-basis of $T_{p}M$. That's why we will prove the existence of such bases.
\begin{thm}\label{t3.1}
If $x=(x^{1},x^{2},x^{3})$ induces a $q$-basis of $T_{p}M$, then for the angles
 $\varphi=\angle(x,qx),\ \phi=\angle(x,q^{2}x)$ and $\theta=\angle(qx,q^{2}x)$
 we have
 \begin{equation}\label{cos}
    \cos\varphi=-\cos\phi=\cos\theta=-\frac{Ba+(A+B)b}{Aa+2B b}\ ,
\end{equation}
 \begin{equation}\label{ugli}
 \varphi\in \Big(\dfrac{\pi}{3},\pi\Big),\ \phi=\theta\in \Big(0, \dfrac{2\pi}{3}\Big)\ ,
 \end{equation}
 where
 \begin{equation}\label{a,b}
    a=(x^{1})^{2}+(x^{2})^{2}+(x^{3})^{2}, \ b=x^{1}x^{2}+x^{1}x^{3}+x^{2}x^{3}.
 \end{equation}
 \end{thm}
\begin{proof}
According to (\ref{f2}), (\ref{ab}), (\ref{f4}) and (\ref{2.1}), we obtain
\begin{equation*}
    \cos\varphi=\dfrac{g(x, qx)}{g(x,x)}=-\frac{Ba+(A+B)b}{A a+2B b}\ ,
\end{equation*}
\begin{equation*}
    \cos\phi=\dfrac{g(x, q^{2}x)}{g(x,x)}=\frac{B a+(A+B)b}{A a+2B b}\ ,
\end{equation*}
\begin{equation*}
    \cos\theta=\dfrac{g(qx, q^{2}x)}{g(x,x)}=-\frac{B a+(A+B)b}{A a+2B b}\ .
\end{equation*}
Then we have (\ref{cos}).

Now we prove (\ref{ugli}). Since $\varphi, \phi$ and $\theta$ are angles between vectors, we have $\varphi, \phi, \theta\in [0, \pi]$. Let us suppose that $\phi\geq \dfrac{2\pi}{3}$. Then, due to (\ref{cos}) and (\ref{a,b}), we find $(x^{1}+x^{2}+x^{3})^{2}\leq 0$, i.e. $x^{1}+x^{2}+x^{3}=0$. The last equality is in contradiction with the inequality (\ref{lema2}). Consequently, we have $\phi <\dfrac{2\pi}{3}$. Since the vectors $x$ and $q^{2}x$ are linearly independent, we have $\phi >0$ and then $\phi\in \Big(0, \dfrac{2\pi}{3}\Big)$, i.e. $\cos\phi\in \Big(-\dfrac{1}{2}, 1\Big)$. Thus, from (\ref{cos}) we get $\cos\varphi\in \Big(-1,\dfrac{1}{2}\Big)$, i.e. $\varphi\in \Big(\dfrac{\pi}{3}, \pi\Big)$.
\end{proof}

\begin{cor} The following assertions are true:
\begin{itemize}
  \item[(i)] If a vector $x=(x^{1},x^{2},x^{3})$ induces a $q$-basis of $T_{p}M$, then this basis is an orthogonal $q$-basis if and only if
      \begin{equation}\label{11}
        Ba+(A+B)b=0.
      \end{equation}

  \item[(ii)] There exist orthogonal $q$-bases in $T_{p}M$.
\end{itemize}
 \end{cor}
 \begin{proof}
The assertion (i) follows immediately from (\ref{cos}), and (ii) follows from (\ref{ugli}).
\end{proof}
\textbf{Example.}
Let $x=\Big(0,-(A+B)+\sqrt{(A-B)(A+3B)},2B\Big)\in T_{p}M$. Then according to (\ref{a,b}) we get
\begin{align*}a&=2(A+B)\Big((A+B)-\sqrt{(A-B)(A+3B)} \Big),\\ b&=-2B\Big((A+B)-\sqrt{(A-B)(A+3B)}\Big).
\end{align*}
These values satisfy the condition (\ref{11}). So $x$ induces an orthogonal $q$-basis $\{x, qx, q^{2}x\}$ of $T_{p}M$.

\section{Some curvature properties of $(M, g, q)$}\label{3}
Let $\nabla$ be the Riemannian connection of the metric $g$ on $(M, g, q)$. The curvature tensor $R$ of $\nabla$ is $R(x, y)z=\nabla_{x}\nabla_{y}z-\nabla_{y}\nabla_{x}z-\nabla_{[x,y]}z$. The tensor of type $(0, 4)$  associated with $R$ is defined as follows
\begin{equation*}
    R(x, y, z, u)=g(R(x, y)z,u).
    \end{equation*}
For the local components of the tensor $R$ of type $(1,3)$ we have \cite{11}:
\begin{equation}\label{R12}
    R_{ijk}^{h}=\partial_{j}\Gamma_{ik}^{h}-\partial_{k}\Gamma_{ij}^{h}+\Gamma_{ik}^{t}\Gamma_{tj}^{h}-\Gamma_{ij}^{t}\Gamma_{tk}^{h},
\end{equation}
 where the Christoffel symbols $\Gamma_{ij}^{h}$ of $\nabla$ are
\begin{equation}\label{2.3}
2\Gamma_{ij}^{h}=g^{th}(\partial_{i}g_{tj}+\partial_{j}g_{ti}-\partial_{t}g_{ij}).
\end{equation}
Here $g^{ij}$ are the elements of the inverse matrix of $(g_{ij})$.

We denote $D=(g_{ij})$ and
\begin{equation}\label{31}
    A_{i}=\dfrac{\partial A}{\partial X^{i}}\ ,\quad B_{i}=\dfrac{\partial B}{\partial X^{i}}\ ,
\end{equation}
where $A$ and $B$ are the functions from (\ref{f2}).
From (\ref{f2}), (\ref{R12}), (\ref{2.3}) and (\ref{31}) by direct calculations we obtain the following
 \begin{prop}\label{propR} The nonzero components of the curvature tensor of type $(0,4)$ of the manifold $(M, g, q)$ are
\begin{equation*}
\begin{split}
R_{1212}&= \frac{1}{2}(2B_{21}-A_{11}-A_{22})\\&+\frac{A+B}{4D}\Big(2A_{3}B_{2}-A_{3}^{2}+(B_{1}-B_{2}-B_{3})(B_{1}+B_{2}-B_{3})\Big)\\&-\frac{B}{4D}\Big(2A_{1}(B_{1}+B_{2}-B_{3})-2B_{2}(B_{1}+B_{2}-B_{3})-2A_{1}A_{3}+2A_{3}B_{2}\Big),
\end{split}
\end{equation*}
\begin{equation*}
\begin{split}
R_{1313}&= \frac{1}{2}(2B_{31}-A_{11}-A_{33})\\&+\frac{A+B}{4D}\Big(2A_{2}B_{3}-A_{2}^{2}+(-B_{1}+B_{2}+B_{3})(-B_{1}+B_{2}-B_{3})\Big)\\&-\frac{B}{4D}\Big(2A_{1}(B_{1}-B_{2}+B_{3})-2B_{3}(B_{1}-B_{2}+B_{3})-2A_{1}A_{2}+2A_{2}B_{3}\Big),
\end{split}
\end{equation*}
\begin{equation*}
\begin{split}
R_{2323}&= \frac{1}{2}(2B_{23}-A_{22}-A_{33})\\&+\frac{A+B}{4D}\Big(2B_{3}A_{1}-A_{1}^{2}+(B_{1}-B_{2}+B_{3})(B_{1}-B_{2}-B_{3})\Big)\\&-\frac{B}{4D}\Big(2A_{2}(B_{2}+B_{3}-B_{1})+2B_{3}(B_{1}-B_{2}-B_{3})-2A_{1}A_{2}+2A_{1}B_{3}\Big),
\end{split}
\end{equation*}
\begin{equation*}
\begin{split}
R_{1213}&= \frac{1}{2}(B_{21}+B_{31}-B_{11}-A_{23})\\&+\frac{A+B}{4D}\Big(A_{1}(B_{2}-B_{3}+B_{1})+2B_{3}(-B_{1}-B_{2}+B_{3})+A_{2}A_{3})\Big)\\&-\frac{B}{4D}\Big(A_{1}^{2}+A_{2}^{2}+A_{3}^{2}+2A_{1}(A_{2}-B_{3})-2A_{3}(B_{1}-B_{3})-2A_{2}B_{3}\\&+(B_{1}-B_{2}-B_{3})(B_{1}+B_{2}-B_{3})\Big),
\end{split}
\end{equation*}
\begin{equation*}
\begin{split}
R_{1223}&=\frac{1}{2}(B_{22}-B_{12}-B_{23}+A_{13})\\&+\frac{A+B}{4D}\Big(A_{2}(B_{2}+B_{3}-B_{1})-(2B_{3}-A_{1})(2B_{2}-A_{3})\Big)\\&-\frac{B}{4D}\Big(-A_{1}^{2}+A_{2}^{2}+A_{3}^{2}+2A_{1}(B_{2}+B_{3})+2A_{2}(B_{2}-B_{3})-4B_{2}B_{3}\\&+2A_{3}(B_{3}-B_{1})+(B_{1}+B_{2}-B_{3})(B_{1}-B_{2}-B_{3})\Big),
\end{split}
\end{equation*}
\begin{equation*}
\begin{split}
R_{1323}&= \frac{1}{2}(B_{23}-B_{33}+B_{13}-A_{12})\\&+\frac{A+B}{4D}\Big((2B_{2}-A_{1})(2B_{3}-A_{2})-A_{3}(-B_{1}+B_{2}+B_{3})\Big)\\&-\frac{B}{4D}\Big(A_{1}^{2}-A_{2}^{2}-A_{3}^{2}-2A_{1}(B_{2}+B_{3})+2A_{2}(B_{1}-B_{2})+4B_{2}B_{3}\\&+2A_{3}(B_{2}-B_{3})-B_{1}+B_{2}+B_{3})(B_{1}-B_{2}+B_{3})\Big).
\end{split}
\end{equation*}
The rest of the nonzero components are obtained by the properties $$R_{ijkh}=R_{khij}, \ R_{ijkh}=-R_{jikh}=-R_{ijhk}.$$
\end{prop}

If $\{x, y\}$ is a non-degenerate $2$-plane spanned by vectors $x, y \in T_{p}M$, then its sectional curvature is
\begin{equation}\label{3.3}
    \mu(x,y)=\frac{R(x, y, x, y)}{g(x, x)g(y, y)-g^{2}(x, y)}\ .
\end{equation}
In this section we obtain some curvature properties of a manifold $(M, g, q)$ on which the following identity is valid
\begin{equation}\label{R}
 R(qx,qy, qz, qu)=R(x, y, z, u).
\end{equation}

\begin{thm}\label{th4}
Let $(M, g, q)$ be a manifold with property \eqref{R}. If a vector $u$ induces a $q$-basis of $T_{p}M$ and $\{x, qx, q^{2}x\}$ is an orthonormal $q$-basis of $T_{p}M$, then
\begin{equation}\label{mu-r}
   \mu(u,qu)-\mu(x, qx)=\frac{2\cos\varphi}{1-\cos\varphi}R(x, qx, x, q^{2}x),
\end{equation}
where $\varphi=\angle(u, qu)$.
\end{thm}
\begin{proof}
 From (\ref{R}) we find
 \begin{equation}\label{Rv1}
R(q^{2}x,q^{2}y, q^{2}z, q^{2}u)= R(qx,qy, qz, qu)=R(x, y, z, u).
\end{equation}
 In (\ref{Rv1}) we substitute $qx$ for $y$, $q^{2}x$ for $z$ and $x$ for $u$. Comparing the obtained results, we get
    \begin{equation}\label{dop2}
    R(x, qx, q^{2}x, x)=R(qx, q^{2}x, x, qx)=R(q^{2}x, x, qx, q^{2}x).
    \end{equation}
 Let $u=\alpha x+\beta qx +\gamma q^{2}x$, where $\alpha,\beta,\gamma \in \mathbb{R}$. Then $qu=-\gamma x+\alpha qx +\beta q^{2}x$.
 Using the linear properties of the curvature tensor $R$, we obtain
 \begin{equation*}
    \begin{split}
    R(u,qu,u,qu)&=(\alpha^{2}+\beta\gamma)^{2}R(x, qx, x, qx)\\&+(\gamma^{2}+\alpha\beta)^{2}R(x, q^{2}x, x, q^{2}x)\\&+(\beta^{2}-\alpha\gamma)^{2}R(qx, q^{2}x, qx, q^{2}x)\\&-2(\alpha^{2}+\beta\gamma)(\gamma^{2}+\alpha\beta)R(x, qx, q^{2}x, x)\\&-2(\gamma^{2}+\alpha\beta)(\beta^{2}-\alpha\gamma)R(q^{2}x, x, qx, q^{2}x)\\&+2(\alpha^{2}+\beta\gamma)(\beta^{2}-\alpha\gamma)R(x, qx, qx, q^{2}x).
    \end{split}
\end{equation*}
Having in mind (\ref{R}) and (\ref{dop2}), from the latter equality we get
 \begin{equation}\label{cor}
    \begin{split}
    &R(u,qu,u,qu)=\\&=\Big((\alpha^{2}+\beta\gamma)^{2}+(\gamma^{2}+\alpha\beta)^{2}+(\beta^{2}-\alpha\gamma)^{2}\Big)R(x, qx, x, qx)\\&+2\Big((\alpha^{2}+\beta\gamma)(-\gamma^{2}-\alpha\beta)+(-\gamma^{2}-\alpha\beta)(\beta^{2}-\alpha\gamma)\\&-(\alpha^{2}+\beta\gamma)(\beta^{2}-\alpha\gamma)\Big)R(x, qx, q^{2}x, x).
    \end{split}
\end{equation}
Since the $q$-basis $\{x, qx, q^{2}x\}$ is orthonormal we have
\begin{equation*}
    g(u, u)= g(qu, qu)=\alpha^{2}+\beta^{2}+\gamma^{2},\quad g(u, qu)= \alpha\beta+\beta\gamma-\gamma\alpha.
\end{equation*}
Due to (\ref{2.1}) and (\ref{3.3}) we get
\begin{equation*}
    \mu(u,qu)=\frac{R(u, qu, u, qu)}{g(u, u)g(qu, qu)-g^{2}(u, qu)}\ .
\end{equation*}
 We can suppose that $g(u, u)=1$. Then we obtain
     \begin{equation}\label{mu3}
       \mu(u,qu)=\frac{R(u, qu, u, qu)}{1-cos^{2}\varphi}
 \end{equation}
 and
\begin{equation*}
   \alpha^{2}+\beta^{2}+\gamma^{2}=1,\quad
   \alpha\beta+\beta\gamma-\gamma\alpha=\cos\varphi.
\end{equation*}
The latter two equations imply
    \begin{align*}
    \cos^{2}\varphi+ \cos\varphi&= (\alpha^{2}+\beta\gamma)(\gamma^{2}+\alpha\beta)+(\gamma^{2}+\alpha\beta)(\beta^{2}-\alpha\gamma)\\&+(\alpha^{2}+\beta\gamma)(\beta^{2}-\alpha\gamma),\\
      1-\cos^{2}\varphi&=(\alpha^{2}+\beta\gamma)^{2}+(\gamma^{2}+\alpha\beta)^{2}+(\beta^{2}-\alpha\gamma)^{2}.
    \end{align*}
Then, using (\ref{cor}) and (\ref{mu3}) we obtain (\ref{mu-r}).
   \end{proof}
   \begin{thm}\label{th6}
Let $(M,g, q)$ be a manifold with property \eqref{R}. If a vector $u$ induces a $q$-basis of $T_{p}M$ and $\{x, qx, q^{2}x\}$ is an orthonormal $q$-basis of $T_{p}M$, then for $x, u, y \in T_{p}M$ with $\angle(y, qy) =\dfrac{2\pi}{3}$ the following equality is valid
\begin{equation}\label{mu-r2}
   \mu(u, qu)=\frac{1+2\cos\varphi}{1-\cos\varphi}\mu(x, qx)-\frac{3\cos\varphi}{1-\cos\varphi}\mu(y, qy),
\end{equation}
where $\varphi=\angle(u, qu)$.
\end{thm}
\begin{proof}
In (\ref{mu-r}) we substitute $y$ for $u$ and because of the condition \\$\angle(y,qy)=\dfrac{2\pi}{3}$ we get
\begin{equation*}
        R(x, qx, x, q^{2}x)=\dfrac{3}{2}\Big(\mu(x, qx)-\mu(y, qy)\Big).
\end{equation*}
 Taking into account the last result and (\ref{mu-r}), we get (\ref{mu-r2}).
\end{proof}

\begin{prop}
Let $(M, g, q)$ be a manifold with property \eqref{R}. If a vector $u$ induces a $q$-basis of $T_{p}M$, then the sectional curvatures of the $2$-planes $\{u, qu\}$, $\{qu, q^{2}u\}$ and $\{q^{2}u, u\}$ are equal.
\end{prop}
\begin{proof}
In (\ref{Rv1})
we substitute $u$ for $x$, $qu$ for $y$, $u$ for $z$ and $qu$ for $u$. Then we get
       \begin{equation}\label{3.7}
    R(u, qu, u, qu)=R(u, q^{2}u, qu, q^{2}u)=R(qu, q^{2}u, qu, q^{2}u).
    \end{equation}
From (\ref{2.1}), (\ref{3.3}) and (\ref{3.7}) for the sectional curvatures of the $2$-planes of a $q$-basis $\{u, qu, q^{2}u\}$ we obtain
    $\mu(u,qu)= \mu(u, q^{2}u)= \mu(qu, q^{2}u)$.
\end{proof}
\begin{prop}
The property \eqref{R} of the manifold $(M, g, q)$ is equivalent to the conditions
\begin{equation}\label{r1=r6}
    R_{1212}=R_{1313}=R_{2323},\quad R_{1213}=R_{1323}=R_{1223},
\end{equation}
where $R_{ijkh}$ are the local components of the curvature tensor $R$ of type $(0, 4)$.
\end{prop}
\begin{proof}
The local form of (\ref{R}) is
\begin{equation}\label{rlocal}
    R_{tslm}q_{i}^{t}q_{j}^{s}q_{k}^{l}q_{h}^{m}=R_{ijkh}.
\end{equation}
From (\ref{rlocal}) and (\ref{f4}) we find
\begin{align}\label{r1=r3}\nonumber
    R_{1212}=R_{2323},\quad R_{1313}=R_{2121},\\ R_{2321}=R_{1213},\quad
    R_{2331}=R_{1223},\\\nonumber R_{2131}=R_{1323},\quad R_{3131}=R_{2323},
\end{align}
which implies (\ref{r1=r6}).

Vice versa, (\ref{r1=r3}) follows from (\ref{r1=r6}). Having in mind (\ref{f4}), we get (\ref{rlocal}).
\end{proof}
\section{A necessary and sufficient condition for $\nabla q=0$}\label{4}

The covariant derivative of the structure $q$ with respect to $\nabla$ of $g$ is \cite{11}
\begin{equation}\label{nabla}
\nabla_{i}q_{j}^{h}=\partial_{i}q_{j}^{h}+\Gamma_{it}^{h}q_{j}^{t}-\Gamma_{ij}^{t}q_{t}^{h}.
\end{equation}
Let the structure $q$ be parallel with respect to the Riemannian connection $\nabla$ of our manifold $(M, g, q)$, i.e. $\nabla q=0$. Then from (\ref{f4}) and (\ref{nabla}) follows
\begin{equation*}
\Gamma_{11}^{3}=\Gamma_{32}^{1},\qquad \Gamma_{11}^{2}=\Gamma_{32}^{2}, \qquad  \Gamma_{11}^{3}=\Gamma_{32}^{3},
\end{equation*}
\begin{equation}\label{gama}
   \  \Gamma_{21}^{2}=\Gamma_{33}^{2},\qquad \Gamma_{21}^{1}=\Gamma_{33}^{1},\qquad \Gamma_{21}^{3}=\Gamma_{33}^{3},
\end{equation}
\begin{equation*}
     \ \Gamma_{22}^{1}=\Gamma_{13}^{1},\qquad \Gamma_{22}^{2}=\Gamma_{13}^{2},\qquad \Gamma_{22}^{3}=\Gamma_{13}^{3}.
\end{equation*}
Using (\ref{f2}), (\ref{2.3}) and (\ref{gama}), we obtain
\begin{equation}\label{30}
    A_{1}=-B_{1}+B_{2}+B_{3},\ A_{2}=B_{1}-B_{2}+B_{3},\ A_{3}=B_{1}+B_{2}-B_{3}.
\end{equation}
The system (\ref{30}) has the matrix form
\begin{equation*}
    (A_{1}, A_{2}, A_{3})= (B_{1}, B_{2}, B_{3})\left(
    \begin{array}{ccc}
    -1 & 1 & 1 \\
       1 & -1 & 1 \\
       1 & 1 & -1 \\
       \end{array}
     \right).
\end{equation*}
Therefore, we have
     \begin{equation}\label{2.5}
    \grad A=\grad B
    \begin{pmatrix}
      -1 & 1 & 1 \\
      1 & -1 & 1 \\
      1 & 1 & -1 \\
    \end{pmatrix},
\end{equation}
where $\grad A$ and $\grad B$ are gradients of the functions $A$ and $B$.

Vice versa, let (\ref{2.5}) be valid. Then (\ref{f2}) and (\ref{2.3}) imply
\begin{align*}
    \Gamma_{11}^{1}=\Gamma_{12}^{2}=\Gamma_{13}^{3}=\Gamma_{22}^{3}=\Gamma_{23}^{1}=\Gamma_{33}^{2}&=\frac{1}{2D}\Big(AA_{1}+B(B_{2}-3B_{1}+B_{3})\Big),\\ \Gamma_{11}^{3}=\Gamma_{12}^{1}=\Gamma_{13}^{2}=\Gamma_{22}^{2}=\Gamma_{23}^{3}=\Gamma_{33}^{1}&=\frac{1}{2D}\Big(AA_{2}+B(B_{3}-3B_{2}+B_{1})\Big),\\
    \Gamma_{11}^{2}=\Gamma_{12}^{3}=\Gamma_{13}^{1}=\Gamma_{22}^{1}=\Gamma_{23}^{2}=\Gamma_{33}^{3}&=\frac{1}{2D}\Big(AA_{3}+B(B_{1}-3B_{3}+B_{2})\Big).
    \end{align*}
Now, we verify that
    $\Gamma_{it}^{h}q_{j}^{.t}=\Gamma_{ij}^{t}q_{t}^{.h}$
is valid. That means
$\nabla_{i}q_{j}^{s}=0$, i.e. $\nabla q=0$.

Thus, it is valid the following
\begin{thm}
 The structure $q$ is parallel with respect to the Riemannian connection $\nabla$ of $g$ on a manifold $(M, g, q)$ if and only if the condition \eqref{2.5} is valid.
 \end{thm}

 \section{Example of a manifold $(M, g, q)$}\label{5}

Let $(M, g, q)$ be a manifold with a metric $g$ determined by (\ref{f2}) and
\begin{equation}\label{36}
    A=2X^{1},\ B=2X^{1}+X^{2}+X^{3},
\end{equation}
where $$2X^{1}+X^{2}+X^{3}>0,\ X^{2}+X^{3}<0.$$
Obviously, the condition (\ref{ab}) is satisfied.
Due to (\ref{31}), (\ref{36}) and  Proposition~\ref{propR} we obtain
   \begin{equation*}
   \begin{split}
       R_{1212} &=R_{1313}=R_{2323}= \frac{2X^{1}+X^{2}+X^{3}}{(X^{2}+X^{3})(6X^{1}+2X^{2}+2X^{3})}\ ,\\
   R_{1213} &=R_{1323}=R_{1223}=0.
    \end{split}
   \end{equation*}
We check directly that the conditions (\ref{r1=r6}) are valid, but the condition (\ref{2.5}) for the functions $A$ and $B$ is not valid.

Therefore, it is valid the following

\begin{prop}
\begin{itemize} The manifold $(M,g, q)$ with a metric $g$ defined by \eqref{36} has the following properties:
  \item[(i)] The curvature identity \eqref{R} is valid.

  \item[(ii)] The structure $q$ is not parallel with respect to the Riemannian connection $\nabla$ of $g$.

\item[(iii)] $M$ is not a flat manifold.
\end{itemize}
\end{prop}

\author{Iva Dokuzova \\Department of
Geometry\\ FMI, University of Plovdiv\\Plovdiv, Bulgaria \\
dokuzova@uni-plovdiv.bg}

\begin{thebibliography}{99}

\bibitem{1}
\textsc{Borisov, A., Ganchev, G.}: Curvature properties of Kaehlerian manifolds
with B-metric. Math. Educ. Math., Proc of 14th Spring Conf. of UBM,
Sunny Beach,  220--226 (1985)

\bibitem{3}
\textsc{Gribachev, K., Mekerov, D., Djelepov, G.}: On the geometry of almost $B$-manifolds
Compt. Rend. Acad. Bulg. Sci. \textbf{38}, 563--566 (1985)

\bibitem{4}
\textsc{Gribachev, K., Djelepov, G.}: On the geometry of the normal generalized $B$-manifolds.
Plovdiv Univ. Sci. Works - Math. \textbf{23}, 157--168 (1985)

\bibitem{8}
\textsc{Davis, P. R.}: Circulant matrices. John Wiley, New York (1979)

\bibitem{5}
\textsc{Dzhelepov, G., Dokuzova, I., Razpopov, D.}: On a three-dimensional Riemannian manifold with an additional structure. Plovdiv Univ. Sci. Works - Math. \textbf{38}, 17--27 (2011)

\bibitem{10}
 \textsc{Reyes, E., Montesinos, A., Gadea, P.M.}: Connections making parallel a metric $(J^{4}=1)$-structure. An. Sti. Univ. ``Al. I. Cuza`` \textbf{28}, 49--54 (1982)

\bibitem{9}
\textsc{Stanilov, G., Filipova, L.}: About the circulate geometry, In: Proceedings
of XVI Geometrical Seminar, Vrnjachka Banja, 96--102 (2010)

\bibitem{6}
\textsc{Yano, K., Ishihara, S.}: Structure defined by $f$ satisfying $f^{3}+f=0$. Proc.US-Japan Seminar in Differential Geometry., Kyoto,  153--166 (1965)

\bibitem{7}
\textsc{Yano, K.}: On a structure defined by a tensor field of type $(1, 1)$ satisfying $f^{3}+f=0$. Tensor, 99--109 (1963)

\bibitem{11}
\textsc{Yano, K.}: Differential geometry on complex and almost complex spaces. Pure and Applied Math. \textbf{49}, New York, Pergamont Press Book, (1965)
\end{thebibliography}
\end{document}